\documentclass{article}

\usepackage[utf8]{inputenc}
\usepackage{amssymb, amsmath}
\usepackage{amsthm}
\usepackage{amsfonts} 
\usepackage{booktabs}

\newtheorem{thm}{Theorem}[section]

\usepackage[parfill]{parskip}

\title{Canonical forms for pairs of commuting nilpotent $4\times 4$ matrices under simultaneous similarity}
\author{Jiuzhao Hua} 

\begin{document}
\date{\vspace{-0.5cm}}\date{} 
\maketitle
\begin{abstract}
We provide a list of canonical forms for all pairs of commuting nilpotent $4\times 4$ matrices over an algebraically closed field under simultaneous similarity.
\end{abstract}

\section{Introduction}
Let $n$ be a non-negative integer. A pair of $n\times n$ matrices $(A,B)$ over a field $\mathbb{F}$ is said to be \textit{similiar} to another pair $(C,D)$ if there exists a non-singular matrix $X$ over $\mathbb{F}$ such that $X^{-1}AX = C$ and $X^{-1}BX = D$.
A pair $(A,B)$ is said to be \textit{decomposable} if $(A,B)$ is similar to a pair which has the following form:
\[
\left(
\left[
\begin{array}{cc}
A_1 & 0  \\ 
0 & A_2 \\
\end{array}
\right],
\left[
\begin{array}{cc}
B_1 & 0  \\ 
0 & B_2 \\
\end{array}
\right]
\right),
\]
where $A_1$ and $B_1$ are square matrices of the same size.

Classifying all such pairs up to simultaneous similarity is a classical `wild' problem. It is believed that there is little hope in finding some algebraic objects that parametrize all similarity classes of
such pairs. Even when the pairs are restricted to commuting nilpotent pairs, i.e., $A^n=0$, $B^n=0$ and $AB=BA$, such a problem is still a `wild' problem. However, when $n$ is small, it is possible 
to find a list of canonical forms for all such pairs.

Let $\textrm{Mat}(n,\mathbb{F})$ be the matrix algebra which consists of all $n\times n$ matrices over $\mathbb{F}$, $\mathrm{GL}(n,\mathbb{F})$ be the General Linear Group which consists of all non-singular $n\times n$ matrices over $\mathbb{F}$. 
For a matrix $A\in\textrm{Mat}(n,\mathbb{F})$, let $\mathrm{Stab}(A)$ be the stabilizer group of $A$ in $\mathrm{GL}(n,\mathbb{F})$, i.e., 
\[
\mathrm{Stab}(A) = \{X\in\mathrm{GL}(n,\mathbb{F}) : X^{-1} A X = A\},
\]
and $\mathrm{NilC}(A)$ be the nilpotent commutator of $A$, i.e., 
\[
\mathrm{NilC}(A) = \{B\in\textrm{Mat}(n,\mathbb{F}) : B^n=0, AB=BA\}.
\]

For historical results on this topic, we refer the interested readers to Baranovsky's paper \cite{VB} and Basili \& Iarrobino's paper \cite{B-I}.
In this paper, we aim to classify all pairs of commuting nilpotent $4\times 4$ matrices under simultaneous similarity. The strategy we use is as follows: first we fix the matrix $A$, for example, we may assume that $A$ has a 
particular Jordan normal form, then we find $\mathrm{Stab}(A)$ and $\mathrm{NilC}(A)$ explicitly, and then we convert the problem of classifying all commuting nilpotent paris $(A,B)$ with $A$ fixed and $B$ arbitrary to the problem of classifying
the orbits of the following group action:
\begin{align*}
\mathrm{Stab}(A) \times \mathrm{NilC}(A) &\to \mathrm{NilC}(A) \\
(X, B)\,\,\,\,\,\,\,\,\,\,\,\,\,\,\,\, &\mapsto  X^{-1}BX.
\end{align*}

\section{The canonical forms}

\begin{thm}
Every indecomposable pair of commuting nilpotent $4\times 4$ matrices over an algebraically closed field $\mathbb{F}$ is similar to a unique pair in one of the following forms, where 
$\lambda, \mu, \nu \in \mathbb{F}$:
\[
(1.1) \,\,
\left(
\left[
\begin{array}{ccccc}
0 & 1 & 0 & 0 \\ 
0 & 0 & 1 & 0 \\ 
0 & 0 & 0 & 1 \\
0 & 0 & 0 & 0 \\
\end{array}
\right], \,\,
\left[
\begin{array}{ccccc}
0 & \lambda & \mu & \nu \\ 
0 & 0 & \lambda & \mu \\ 
0 & 0 & 0 & \lambda \\
0 & 0 & 0 & 0 \\
\end{array}
\right]
\right),
\]
\[(2.1) \,\,
\newcommand*{\temp}{\multicolumn{1}{|}{}}
\left(
\left[
\begin{array}{cccccc}
0 & 1 & 0 & \temp & 0 \\ 
0 & 0 & 1 & \temp & 0 \\ 
0 & 0 & 0 & \temp & 0 \\
\cline{1-5}
0 & 0 & 0 & \temp & 0 \\
\end{array}
\right], \,\,
\left[
\begin{array}{cccccc}
0 & \lambda & 0 & \temp & \mu \\ 
0 & 0 & \lambda & \temp & 0 \\ 
0 & 0 & 0 & \temp & 0 \\
\cline{1-5}
0 & 0 & 1 & \temp & 0 \\
\end{array}
\right]
\right),
\]
\[(2.2) \,\,
\newcommand*{\temp}{\multicolumn{1}{|}{}}
\left(
\left[
\begin{array}{cccccc}
0 & 1 & 0 & \temp & 0 \\ 
0 & 0 & 1 & \temp & 0 \\ 
0 & 0 & 0 & \temp & 0 \\
\cline{1-5}
0 & 0 & 0 & \temp & 0 \\
\end{array}
\right], \,\,
\left[
\begin{array}{cccccc}
0 & \lambda & 0 & \temp & 1 \\ 
0 & 0 & \lambda & \temp & 0 \\ 
0 & 0 & 0 & \temp & 0 \\
\cline{1-5}
0 & 0 & 0 & \temp & 0 \\
\end{array}
\right]
\right),
\]
\[(3.1) \,\,
\newcommand*{\temp}{\multicolumn{1}{|}{}}
\left(
\left[
\begin{array}{cccccc}
0 & 1 & \temp & 0 & 0 \\ 
0 & 0 & \temp & 0 & 0 \\ 
\cline{1-5}
0 & 0 & \temp & 0 & 1 \\
0 & 0 & \temp & 0 & 0 \\
\end{array}
\right], \,\,
\left[
\begin{array}{cccccc}
0 & 0 & \temp & 1 & 0 \\ 
0 & 0 & \temp & 0 & 1 \\ 
\cline{1-5}
0 & \lambda & \temp & 0 & \mu \\
0 & 0 & \temp & 0 & 0 \\
\end{array}
\right]
\right),
\]
\[(3.2) \,\,
\newcommand*{\temp}{\multicolumn{1}{|}{}}
\left(
\left[
\begin{array}{cccccc}
0 & 1 & \temp & 0 & 0 \\ 
0 & 0 & \temp & 0 & 0 \\ 
\cline{1-5}
0 & 0 & \temp & 0 & 1 \\
0 & 0 & \temp & 0 & 0 \\
\end{array}
\right], \,\,
\left[
\begin{array}{cccccc}
0 & \lambda & \temp & 0 & 1 \\ 
0 & 0 & \temp & 0 & 0 \\ 
\cline{1-5}
0 & 0 & \temp & 0 & \lambda \\
0 & 0 & \temp & 0 & 0 \\
\end{array}
\right]
\right),
\]
\[(4.1) \,\,
\newcommand*{\temp}{\multicolumn{1}{|}{}}
\left(
\left[
\begin{array}{cccccc}
0 & 1 & \temp & 0 & 0 \\ 
0 & 0 & \temp & 0 & 0 \\ 
\cline{1-5}
0 & 0 & \temp & 0 & 0 \\
0 & 0 & \temp & 0 & 0 \\
\end{array}
\right], \,\,
\left[
\begin{array}{cccccc}
0 & 0 & \temp & \lambda & 0 \\ 
0 & 0 & \temp & 0 & 0 \\ 
\cline{1-5}
0 & 0 & \temp & 0 & 1 \\
0 & 1 & \temp & 0 & 0 \\
\end{array}
\right]
\right),
\]
\[(4.2) \,\,
\newcommand*{\temp}{\multicolumn{1}{|}{}}
\left(
\left[
\begin{array}{cccccc}
0 & 1 & \temp & 0 & 0 \\ 
0 & 0 & \temp & 0 & 0 \\ 
\cline{1-5}
0 & 0 & \temp & 0 & 0 \\
0 & 0 & \temp & 0 & 0 \\
\end{array}
\right], \,\,
\left[
\begin{array}{cccccc}
0 & 0 & \temp & 1 & 0 \\ 
0 & 0 & \temp & 0 & 0 \\ 
\cline{1-5}
0 & 0 & \temp & 0 & 1 \\
0 & 0 & \temp & 0 & 0 \\
\end{array}
\right]
\right),
\]
\[(4.3) \,\,
\newcommand*{\temp}{\multicolumn{1}{|}{}}
\left(
\left[
\begin{array}{cccccc}
0 & 1 & \temp & 0 & 0 \\ 
0 & 0 & \temp & 0 & 0 \\ 
\cline{1-5}
0 & 0 & \temp & 0 & 0 \\
0 & 0 & \temp & 0 & 0 \\
\end{array}
\right], \,\,
\left[
\begin{array}{cccccc}
0 & 0 & \temp & 1 & 0 \\ 
0 & 0 & \temp & 0 & 0 \\ 
\cline{1-5}
0 & 0 & \temp & 0 & 0 \\
0 & 1 & \temp & 0 & 0 \\
\end{array}
\right]
\right),
\]
\[(5.1) \,\,
\newcommand*{\temp}{\multicolumn{1}{|}{}}
\left(
\left[
\begin{array}{ccccc}
0 & 0 & 0 & 0 \\ 
0 & 0 & 0 & 0 \\ 
0 & 0 & 0 & 0 \\
0 & 0 & 0 & 0 \\
\end{array}
\right], \,\,
\left[
\begin{array}{ccccc}
0 & 1 & 0 & 0 \\ 
0 & 0 & 1 & 0 \\ 
0 & 0 & 0 & 1 \\
0 & 0 & 0 & 0 \\
\end{array}
\right]
\right).
\]
\end{thm}
\begin{proof}
Assuming that $(A,B)$ is an indecomposable pair of commuting nilpotent  $4\times 4$ matrices over $\mathbb{F}$, 
the Jordan normal form theorem implies that $A$ is similar to one of the following matrices:
\[
\left[
\begin{array}{ccccc}
0 & 1 & 0 & 0 \\ 
0 & 0 & 1 & 0 \\ 
0 & 0 & 0 & 1 \\
0 & 0 & 0 & 0 \\
\end{array}
\right], 
\left[
\begin{array}{ccccc}
0 & 1 & 0 & 0 \\ 
0 & 0 & 1 & 0 \\ 
0 & 0 & 0 & 0 \\
0 & 0 & 0 & 0 \\
\end{array}
\right], 
\left[
\begin{array}{ccccc}
0 & 1 & 0 & 0 \\ 
0 & 0 & 0 & 0 \\ 
0 & 0 & 0 & 1 \\
0 & 0 & 0 & 0 \\
\end{array}
\right],
\]
\[
\left[
\begin{array}{ccccc}
0 & 1 & 0 & 0 \\ 
0 & 0 & 0 & 0 \\ 
0 & 0 & 0 & 0 \\
0 & 0 & 0 & 0 \\
\end{array}
\right],
\left[
\begin{array}{ccccc}
0 & 0 & 0 & 0 \\ 
0 & 0 & 0 & 0 \\ 
0 & 0 & 0 & 0 \\
0 & 0 & 0 & 0 \\
\end{array}
\right].
\]

\textbf{Case 1.} Let 
\[
A = \left[
\begin{array}{ccccc}
0 & 1 & 0 & 0 \\ 
0 & 0 & 1 & 0 \\ 
0 & 0 & 0 & 1 \\
0 & 0 & 0 & 0 \\
\end{array}
\right].
\]
Then we have,
\[
\mathrm{Stab}(A) = \left\{
\left[
\begin{array}{ccccc}
x & y & z & w \\ 
0 & x & y & z \\ 
0 & 0 & x & y \\
0 & 0 & 0 & x \\
\end{array}
\right] :
x, y, z, w \in \mathbb{F} \textrm{ and } x\ne 0
\right\},
\]
\[
\mathrm{NilC}(A) = \left\{
\left[
\begin{array}{ccccc}
0 & \lambda & \mu & \nu \\ 
0 & 0 & \lambda & \mu \\ 
0 & 0 & 0 & \lambda \\
0 & 0 & 0 & 0 \\
\end{array}
\right] :
\lambda, \mu, \nu \in \mathbb{F}
\right\}.
\]
Since for any $B\in\mathrm{NilC}(A)$ and any $X\in\mathrm{Stab}(A)$, we have $X^{-1}BX = B$, and hence there is exactly one element in each orbit of $\mathrm{NilC}(A)$. 
$(A,B)$ is indecomposable because $A$ is indecomposable. This verifies (1.1).

\textbf{Case 2.} Let 
\[
\newcommand*{\temp}{\multicolumn{1}{|}{}}
A = \left[
\begin{array}{cccccc}
0 & 1 & 0 & \temp & 0 \\ 
0 & 0 & 1 & \temp & 0 \\ 
0 & 0 & 0 & \temp & 0 \\
\cline{1-5}
0 & 0 & 0 & \temp & 0 \\
\end{array}
\right].
\]
Then, by Theorem 2.1 of \cite{D-H}, we have
\[
\newcommand*{\temp}{\multicolumn{1}{|}{}}
\mathrm{Stab}(A) = \left\{
\left[
\begin{array}{cccccc}
x & y & z & \temp & s \\ 
0 & x & y & \temp & 0 \\ 
0 & 0 & x & \temp & 0 \\
\cline{1-5}
0 & 0 & t & \temp & w \\
\end{array}
\right] :
x,y,z,w,s,t \in \mathbb{F} \textrm{ and } x\ne 0, w\ne 0
\right\},
\]
\[
\newcommand*{\temp}{\multicolumn{1}{|}{}}
\mathrm{NilC}(A) = \left\{
\left[
\begin{array}{cccccc}
0 & \alpha & \beta & \temp & \sigma \\ 
0 & 0 & \alpha & \temp & 0 \\ 
0 & 0 & 0 & \temp & 0 \\
\cline{1-5}
0 & 0 & \tau & \temp & 0 \\
\end{array}
\right] :
\alpha, \beta, \sigma, \tau \in \mathbb{F}
\right\}.
\]

Let 
\[
\newcommand*{\temp}{\multicolumn{1}{|}{}}
B = \left[
\begin{array}{cccccc}
0 & \alpha & \beta & \temp & \sigma \\ 
0 & 0 & \alpha & \temp & 0 \\ 
0 & 0 & 0 & \temp & 0 \\
\cline{1-5}
0 & 0 & \tau & \temp & 0 \\
\end{array}
\right] \in \mathrm{NilC}(A) 
\text{ and }
X = \left[
\begin{array}{cccccc}
x & y & z & \temp & s \\ 
0 & x & y & \temp & 0 \\ 
0 & 0 & x & \temp & 0 \\
\cline{1-5}
0 & 0 & t & \temp & w \\
\end{array}
\right] \in \mathrm{Stab}(A).
\]
Then, we have
\[
\newcommand*{\temp}{\multicolumn{1}{|}{}}
X^{-1} = \left[
\begin{array}{cccccc}
x^{-1} & -x^{-2}y & * & \temp & -(xw)^{-1}s \\ 
0 & x^{-1} & -x^{-2}y & \temp & 0 \\ 
0 & 0 & x^{-1} & \temp & 0 \\
\cline{1-5}
0 & 0 & -(xw)^{-1}t & \temp & w^{-1} \\
\end{array}
\right] ,
\]
where the $*$ represents a value that we do not bother to know.
Thus, we have
\[
\newcommand*{\temp}{\multicolumn{1}{|}{}}
X^{-1}BX = 
\left[
\begin{array}{cccccc}
0 & \alpha & \beta'  & \temp & x^{-1}w\sigma \\ 
0 & 0 & \alpha & \temp & 0 \\ 
0 & 0 & 0 & \temp & 0 \\
\cline{1-5}
0 & 0 & xw^{-1}\tau & \temp & 0 \\
\end{array}
\right],
\]
where $\beta' = \beta -w^{-1}s\tau + x^{-1}t\sigma$.

$\sigma$ and $\tau$ cannot be $0$ at the same time, otherwise $(A,B)$ is decomposable. If $\tau\ne 0$, then we can reduce $xw^{-1}\tau$
to $1$. At the same time, $x^{-1}w\sigma$ is reduced to $\sigma\tau$. If $\tau=0$ then $\sigma$ cannot be $0$, thus $x^{-1}w\sigma$ can be reduced to $1$.
Since $\sigma$ and $\tau$ cannot be $0$ at the same time, $\beta'$ can be reduced to $0$ because $s$ and $t$ are free. Thus $B$ can be reduced to
\[
\newcommand*{\temp}{\multicolumn{1}{|}{}}
\left[
\begin{array}{cccccc}
0 & \alpha & 0  & \temp & \sigma\tau \\ 
0 & 0 & \alpha & \temp & 0 \\ 
0 & 0 & 0 & \temp & 0 \\
\cline{1-5}
0 & 0 & 1 & \temp & 0 \\
\end{array}
\right] 
\text{ or }
\left[
\begin{array}{cccccc}
0 & \alpha & 0  & \temp & 1 \\ 
0 & 0 & \alpha & \temp & 0 \\ 
0 & 0 & 0 & \temp & 0 \\
\cline{1-5}
0 & 0 & 0 & \temp & 0 \\
\end{array}
\right] .
\]
In any case, $(A,B)$ is indecomposable. This fact is left as an exercise to the interested readers. This verifies (2.1) and (2.2).

\textbf{Case 3.} Let 
\[
\newcommand*{\temp}{\multicolumn{1}{|}{}}
A = \left[
\begin{array}{cccccc}
0 & 1 & \temp & 0 & 0 \\ 
0 & 0 & \temp & 0 & 0 \\ 
\cline{1-5}
0 & 0 & \temp & 0 & 1 \\
0 & 0 & \temp & 0 & 0 \\
\end{array}
\right].
\]

We transform $A$ into another shape so that it is easier to work with. Let
\[
\newcommand*{\temp}{\multicolumn{1}{|}{}}
T = \left[
\begin{array}{ccccc}
1 & 0 & 0 & 0 \\ 
0 & 0 & 1 & 0 \\ 
0 & 1 & 0 & 0 \\
0 & 0 & 0 & 1 \\
\end{array}
\right],
\]
and $A' = T^{-1}AT$, then we have
\[
\newcommand*{\temp}{\multicolumn{1}{|}{}}
A' = \left[
\begin{array}{cccccc}
0 & 0 & \temp & 1 & 0 \\ 
0 & 0 & \temp & 0 & 1 \\ 
\cline{1-5}
0 & 0 & \temp & 0 & 0 \\
0 & 0 & \temp & 0 & 0 \\
\end{array}
\right].
\]
And so we have
\[
\newcommand*{\temp}{\multicolumn{1}{|}{}}
\mathrm{Stab}(A') = \left\{
\left[
\begin{array}{cccccc}
x & s & \temp & y & t \\ 
u & z & \temp & v & w \\ 
\cline{1-5}
0 & 0 & \temp & x & s \\
0 & 0 & \temp & u & z \\
\end{array}
\right] :
\left[
\begin{array}{cc}
x & s \\ 
u & z \\ 
\end{array}
\right] 
\!\in\mathrm{GL}(2,\mathbb{F}),
\left[
\begin{array}{cc}
y & t \\ 
v & w \\ 
\end{array}
\right] 
\!\in\mathrm{Mat}(2,\mathbb{F})
\right\},
\]
\[
\newcommand*{\temp}{\multicolumn{1}{|}{}}
\mathrm{NilC}(A') = \left\{
\left[
\begin{array}{cccccc}
\alpha & \lambda & \temp & \beta & \mu \\ 
\sigma & \gamma & \temp & \tau & \delta \\ 
\cline{1-5}
0 & 0 & \temp & \alpha & \lambda \\
0 & 0 & \temp & \sigma & \gamma \\
\end{array}
\right] :
\left[
\begin{array}{cc}
\alpha & \lambda \\ 
\sigma & \gamma \\ 
\end{array}
\right]^2 = 0,
\left[
\begin{array}{cc}
\beta & \mu \\ 
\tau & \delta \\ 
\end{array}
\right] 
\!\in\mathrm{Mat}(2,\mathbb{F})
\right\}.
\]

Note that, $B\in\mathrm{NilC}(A')$ if and only if $TBT^{-1}\in\mathrm{NilC}(A)$. Let 
\[
\newcommand*{\temp}{\multicolumn{1}{|}{}}
B= \left[
\begin{array}{cccccc}
\alpha & \lambda & \temp & \beta & \mu \\ 
\sigma & \gamma & \temp & \tau & \delta \\ 
\cline{1-5}
0 & 0 & \temp & \alpha & \lambda \\
0 & 0 & \temp & \sigma & \gamma \\
\end{array}
\right] \in \mathrm{Stab}(A')
\textrm{ and }
X = 
\left[
\begin{array}{cccccc}
x & s & \temp & y & t \\ 
u & z & \temp & v & w \\ 
\cline{1-5}
0 & 0 & \temp & x & s \\
0 & 0 & \temp & u & z \\
\end{array}
\right]  \in \mathrm{NilC}(A').
\]
Then we have
\[
\newcommand*{\temp}{\multicolumn{1}{|}{}}
X^{-1}BX = 
\left[
\begin{array}{cccccc}
\alpha' & \lambda' & \temp & * & * \\ 
\sigma' & \gamma' & \temp & * & * \\ 
\cline{1-5}
0 & 0 & \temp & \alpha' & \lambda' \\
0 & 0 & \temp & \sigma' & \gamma' \\
\end{array}
\right],
\]
where 
\[
\newcommand*{\temp}{\multicolumn{1}{|}{}}
\left[
\begin{array}{cc}
\alpha' & \lambda' \\ 
\sigma' & \gamma' \\ 
\end{array}
\right] = 
\left[
\begin{array}{cc}
x & s \\ 
u & z \\ 
\end{array}
\right]^{-1}
\left[
\begin{array}{cc}
\alpha & \lambda \\ 
\sigma & \gamma \\ 
\end{array}
\right]
\left[
\begin{array}{cc}
x & s \\ 
u & z \\ 
\end{array}
\right].
\]
Thus, by the Jordan normal form theorem, $B$ can be reduced to the following matrices:
\[
\newcommand*{\temp}{\multicolumn{1}{|}{}}
\left[
\begin{array}{cccccc}
0 & 1 & \temp & * & * \\ 
0 & 0 & \temp & * & * \\ 
\cline{1-5}
0 & 0 & \temp & 0 & 1 \\
0 & 0 & \temp & 0 & 0 \\
\end{array}
\right]
\text{ or }
\left[
\begin{array}{cccccc}
0 & 0 & \temp & * & * \\ 
0 & 0 & \temp & * & * \\ 
\cline{1-5}
0 & 0 & \temp & 0 & 0 \\
0 & 0 & \temp & 0 & 0 \\
\end{array}
\right].
\]
Let 
\[
\newcommand*{\temp}{\multicolumn{1}{|}{}}
B= \left[
\begin{array}{cccccc}
0 & 1 & \temp & \beta & \mu \\ 
0 & 0 & \temp & \tau & \delta \\ 
\cline{1-5}
0 & 0 & \temp & 0 & 1 \\
0 & 0 & \temp & 0 & 0 \\
\end{array}
\right]
\text{ and }
X = 
\left[
\begin{array}{cccccc}
x & s & \temp & y & t \\ 
0 & x & \temp & v & w \\ 
\cline{1-5}
0 & 0 & \temp & x & s \\
0 & 0 & \temp & 0 & x \\
\end{array}
\right] .
\]
Note that the choices of $X$ are restricted because we do not want to disturb the shape of $B$. We can assume that $x=1$, otherwise we can replace $X$ by $x^{-1}X$, thus we have
\[
\newcommand*{\temp}{\multicolumn{1}{|}{}}
X^{-1} = 
\left[
\begin{array}{cccccc}
1 & -s & \temp & -y+vs & * \\ 
0 & 1 & \temp & -v & -w+vs \\ 
\cline{1-5}
0 & 0 & \temp & 1 & -s \\
0 & 0 & \temp & 0 & 1 \\
\end{array}
\right].
\]
And so,
\[
\newcommand*{\temp}{\multicolumn{1}{|}{}}
X^{-1}BX = 
\left[
\begin{array}{cccccc}
0 & 1 & \temp & v+\beta -s\tau & \mu' \\ 
0 & 0 & \temp & \tau & s\tau + \delta -v \\ 
\cline{1-5}
0 & 0 & \temp & 0 & 1 \\
0 & 0 & \temp & 0 & 0 \\
\end{array}
\right],
\]
where $\mu'=w+s(\beta-s\tau) + \mu -s\delta - y +vs$. $\mu'$ can be reduced to 0 because $y$ is free.
Thus, 
$B$ can be reduced to the following matrix:
\[
\newcommand*{\temp}{\multicolumn{1}{|}{}}
\left[
\begin{array}{cccccc}
0 & 1 & \temp & 0 & 0 \\ 
0 & 0 & \temp & \tau & \beta + \delta \\ 
\cline{1-5}
0 & 0 & \temp & 0 & 1 \\
0 & 0 & \temp & 0 & 0 \\
\end{array}
\right].
\]

Since $(A, TBT^{-1}) = T(A', B)T^{-1}$, $(A, TBT^{-1})$ is similar to the following pair:
\[
\newcommand*{\temp}{\multicolumn{1}{|}{}}
\left(
\left[
\begin{array}{cccccc}
0 & 1 & \temp & 0 & 0 \\ 
0 & 0 & \temp & 0 & 0 \\ 
\cline{1-5}
0 & 0 & \temp & 0 & 1 \\
0 & 0 & \temp & 0 & 0 \\
\end{array}
\right],
\left[
\begin{array}{cccccc}
0 & 0 & \temp & 1 & 0 \\ 
0 & 0 & \temp & 0 & 1 \\ 
\cline{1-5}
0 & \tau & \temp & 0 & \beta + \delta \\
0 & 0 & \temp & 0 & 0 \\
\end{array}
\right]
\right).
\]
The fact that the above pair is always indecomposable is left as an exercise to the interested readers.
This verifies (3.1).

Now let 
\[
\newcommand*{\temp}{\multicolumn{1}{|}{}}
B= \left[
\begin{array}{cccccc}
0 & 0 & \temp & \beta & \mu \\ 
0 & 0 & \temp & \tau & \delta \\ 
\cline{1-5}
0 & 0 & \temp & 0 & 0 \\
0 & 0 & \temp & 0 & 0 \\
\end{array}
\right]
\text{ and }
X = 
\left[
\begin{array}{cccccc}
x & s & \temp & y & t \\ 
u & z & \temp & v & w \\ 
\cline{1-5}
0 & 0 & \temp & x & s \\
0 & 0 & \temp & u & z \\
\end{array}
\right].
\]

Then we have
\[
\newcommand*{\temp}{\multicolumn{1}{|}{}}
X^{-1}BX = 
\left[
\begin{array}{cccccc}
0 & 0 & \temp & \beta' & \mu' \\ 
0 & 0 & \temp & \tau' & \delta' \\ 
\cline{1-5}
0 & 0 & \temp & 0 & 0 \\
0 & 0 & \temp & 0 & 0 \\
\end{array}
\right], 
\]
where 
\[
\left[
\begin{array}{cc}
\beta' & \mu' \\ 
\tau' & \delta' \\ 
\end{array}
\right] 
= \left[
\begin{array}{cc}
x & s \\ 
u & z \\ 
\end{array}
\right] ^{-1}
\left[
\begin{array}{cc}
\beta & \mu \\ 
\tau & \delta \\ 
\end{array}
\right] 
 \left[
\begin{array}{cc}
x & s \\ 
u & z \\ 
\end{array}
\right].
\]

By the Jordan normal form theorem, $B$ can be reduced to
\[
\newcommand*{\temp}{\multicolumn{1}{|}{}}
\left[
\begin{array}{cccccc}
0 & 0 & \temp & \lambda & 0 \\ 
0 & 0 & \temp & 0 & \alpha \\ 
\cline{1-5}
0 & 0 & \temp & 0 & 0 \\
0 & 0 & \temp & 0 & 0 \\
\end{array}
\right]
\text{ or }
\left[
\begin{array}{cccccc}
0 & 0 & \temp & \lambda & 1 \\ 
0 & 0 & \temp & 0 & \lambda \\ 
\cline{1-5}
0 & 0 & \temp & 0 & 0 \\
0 & 0 & \temp & 0 & 0 \\
\end{array}
\right], 
\]
$B$ cannot take the first form above, otherwise $(A, TBT^{-1})$ is decomposable. Thus $B$ must take the second form above, and so $(A, TBT^{-1})$ has the following form:
\[
\newcommand*{\temp}{\multicolumn{1}{|}{}}
\left(
\left[
\begin{array}{cccccc}
0 & 1 & \temp & 0 & 0 \\ 
0 & 0 & \temp & 0 & 0 \\ 
\cline{1-5}
0 & 0 & \temp & 0 & 1 \\
0 & 0 & \temp & 0 & 0 \\
\end{array}
\right], \,\,
\left[
\begin{array}{cccccc}
0 & \lambda & \temp & 0 & 1 \\ 
0 & 0 & \temp & 0 & 0 \\ 
\cline{1-5}
0 & 0 & \temp & 0 & \lambda \\
0 & 0 & \temp & 0 & 0 \\
\end{array}
\right]
\right).
\]
The fact that the above pair is always indecomposable is left as an exercise to the interested readers.
This verifies (3.2).

\textbf{Case 4.} Let 
\[
A =
\newcommand*{\temp}{\multicolumn{1}{|}{}}
\left[
\begin{array}{cccccc}
0 & 1 & \temp & 0 & 0 \\ 
0 & 0 & \temp & 0 & 0 \\ 
\cline{1-5}
0 & 0 & \temp & 0 & 0 \\
0 & 0 & \temp & 0 & 0 \\
\end{array}
\right].
\]
Then, by Theorem 2.1 of \cite{D-H}, we have
\[
\newcommand*{\temp}{\multicolumn{1}{|}{}}
\mathrm{Stab}(A) = 
\left\{
\left[
\begin{array}{cccccc}
x & y & \temp & s & t \\ 
0 & x & \temp & 0 & 0 \\ 
\cline{1-5}
0 & p & \temp & z & u \\
0 & q & \temp & v & w \\
\end{array}
\right] : 
x \ne 0,
\left[
\begin{array}{cc}
z & u \\ 
v & w \\ 
\end{array}
\right] \in\mathrm{GL}(2,\mathbb{F})
\right\},
\]
\[
\newcommand*{\temp}{\multicolumn{1}{|}{}}
\mathrm{NilC}(A) = \left\{
\left[
\begin{array}{cccccc}
0 & \alpha & \temp & \sigma & \tau \\ 
0 & 0 & \temp & 0 & 0 \\ 
\cline{1-5}
0 & \lambda & \temp & \beta & \gamma \\
0 & \mu & \temp & \delta & \eta \\
\end{array}
\right] :
\left[
\begin{array}{cc}
\beta & \gamma \\ 
\delta & \eta \\ 
\end{array}
\right]^2 = 0
\right\}.
\]
Let 
\[
\newcommand*{\temp}{\multicolumn{1}{|}{}}
B = 
\left[
\begin{array}{cccccc}
0 & \alpha & \temp & \sigma & \tau \\ 
0 & 0 & \temp & 0 & 0 \\ 
\cline{1-5}
0 & \lambda & \temp & \beta & \gamma \\
0 & \mu & \temp & \delta & \eta \\
\end{array}
\right] \in \mathrm{NilC}(A) 
\text{ and }
X = \left[
\begin{array}{cccccc}
x & y & \temp & s & t \\ 
0 & x & \temp & 0 & 0 \\ 
\cline{1-5}
0 & p & \temp & z & u \\
0 & q & \temp & v & w \\
\end{array}
\right] \in\mathrm{Stab}(A).
\]
Then we have
\[
\newcommand*{\temp}{\multicolumn{1}{|}{}}
X^{-1}BX = 
\left[
\begin{array}{cccccc}
0 & * & \temp & * & * \\ 
0 & 0 & \temp & 0 & 0 \\ 
\cline{1-5}
0 & * & \temp & \beta' & \gamma' \\
0 & * & \temp & \delta' & \eta' \\
\end{array}
\right],
\]
where 
\[
\newcommand*{\temp}{\multicolumn{1}{|}{}}
\left[
\begin{array}{cc}
\beta' & \gamma' \\ 
\delta' & \eta' \\ 
\end{array}
\right] = 
\left[
\begin{array}{cc}
z & u \\ 
v & w \\ 
\end{array}
\right]^{-1}
\left[
\begin{array}{cc}
\beta & \gamma \\ 
\delta & \eta \\ 
\end{array}
\right]
\left[
\begin{array}{cc}
z & u \\ 
v & w \\ 
\end{array}
\right].
\]
Thus, by the Jordan normal form theorem, $X^{-1}BX$ can be reduced to 
\[
\newcommand*{\temp}{\multicolumn{1}{|}{}}
\left[
\begin{array}{cccccc}
0 & * & \temp & * & * \\ 
0 & 0 & \temp & 0 & 0 \\ 
\cline{1-5}
0 & * & \temp & 0 & 1 \\
0 & * & \temp & 0 & 0 \\
\end{array}
\right]
\text{ or }
\left[
\begin{array}{cccccc}
0 & * & \temp & * & * \\ 
0 & 0 & \temp & 0 & 0 \\ 
\cline{1-5}
0 & * & \temp & 0 & 0 \\
0 & * & \temp & 0 & 0 \\
\end{array}
\right].
\]

Let 
\[
\newcommand*{\temp}{\multicolumn{1}{|}{}}
B = 
\left[
\begin{array}{cccccc}
0 & \alpha & \temp & \sigma & \tau \\ 
0 & 0 & \temp & 0 & 0 \\ 
\cline{1-5}
0 & \lambda & \temp & 0 & 1 \\
0 & \mu & \temp & 0 & 0 \\
\end{array}
\right] 
\text{ and }
X = \left[
\begin{array}{cccccc}
x & y & \temp & s & t \\ 
0 & x & \temp & 0 & 0 \\ 
\cline{1-5}
0 & p & \temp & z & u \\
0 & q & \temp & 0 & z \\
\end{array}
\right] .
\]
We can assume that $x=1$, otherwise we can replace $X$ with $x^{-1}X$. Then we have
\[
\newcommand*{\temp}{\multicolumn{1}{|}{}}
X^{-1} = 
\left[
\begin{array}{cccccc}
1 & * & \temp & -sz^{-1} & sz^{-2}u - tz^{-1} \\ 
0 & 0 & \temp & 0 & 0 \\ 
\cline{1-5}
0 & -z^{-1}p + z^{-2}uq & \temp & z^{-1} & -z^{-2}u \\
0 & -z^{-1}q & \temp & 0 & z^{-1} \\
\end{array}
\right].
\]
Thus we have
\[
\newcommand*{\temp}{\multicolumn{1}{|}{}}
X^{-1}BX =
\left[
\begin{array}{cccccc}
0 & \alpha' & \temp & \sigma z & \sigma u + \tau z - s \\ 
0 & 0 & \temp & 0 & 0 \\ 
\cline{1-5}
0 & z^{-1}(\lambda - z^{-1}u\mu +q) & \temp & 0 & 1 \\
0 & z^{-1}\mu & \temp & 0 & 0 \\
\end{array}
\right] ,
\]
where $\alpha' = \alpha -sz^{-1}\lambda +(sz^{-2}u - tz^{-1})\mu + \sigma p + (\tau-sz^{-1})q$.

$\mu$ and $\sigma$ cannot be 0 at the same time, otherwise $B$ can be reduced to the following form:
\[
\newcommand*{\temp}{\multicolumn{1}{|}{}}
\left[
\begin{array}{cccccc}
0 & * & \temp & 0 & 0 \\ 
0 & 0 & \temp & 0 & 0 \\ 
\cline{1-5}
0 & 0 & \temp & 0 & 1 \\
0 & 0 & \temp & 0 & 0 \\
\end{array}
\right],
\]
which implies that $(A, B)$ is decomposable. As such, $\alpha'$ can always be reduced to 0 because of the term $-z^{-1}\mu t + \sigma p$ in $\alpha'$.

If $\mu\ne 0$ then $B$ can be reduced to the following matrix:
\[
\newcommand*{\temp}{\multicolumn{1}{|}{}}
\left[
\begin{array}{cccccc}
0 & 0 & \temp & \sigma\mu & 0 \\ 
0 & 0 & \temp & 0 & 0 \\ 
\cline{1-5}
0 & 0 & \temp & 0 & 1 \\
0 & 1 & \temp & 0 & 0 \\
\end{array}
\right].
\]
If $\mu=0$ then $\sigma\ne 0$, thus $B$ can be reduced to the following matrix:
\[
\newcommand*{\temp}{\multicolumn{1}{|}{}}
\left[
\begin{array}{cccccc}
0 & 0 & \temp & 1 & 0 \\ 
0 & 0 & \temp & 0 & 0 \\ 
\cline{1-5}
0 & 0 & \temp & 0 & 1 \\
0 & 0 & \temp & 0 & 0 \\
\end{array}
\right].
\]
In any case, $(A,B)$ is indecomposable. This fact is left as an exercise to the interested readers. This verifies (4.1) and (4.2).

Now, let
\[
\newcommand*{\temp}{\multicolumn{1}{|}{}}
B = 
\left[
\begin{array}{cccccc}
0 & \alpha & \temp & \sigma & \tau \\ 
0 & 0 & \temp & 0 & 0 \\ 
\cline{1-5}
0 & \lambda & \temp & 0 & 0 \\
0 & \mu & \temp & 0 & 0 \\
\end{array}
\right] 
\text{ and }
X = \left[
\begin{array}{cccccc}
x & y & \temp & s & t \\ 
0 & x & \temp & 0 & 0 \\ 
\cline{1-5}
0 & p & \temp & z & u \\
0 & q & \temp & v & w \\
\end{array}
\right] .
\]
We can assume that $x=1$ otherwise we can replace $X$ with $x^{-1}X$. Then we have
\[\def\arraystretch{1.1}
\newcommand*{\temp}{\multicolumn{1}{|}{}}
X^{-1} = 
\left[
\begin{array}{ccc}
	\left[
	\begin{array}{cc}
	1 & * \\ 
	0 & 1 \\ 
	\end{array}
	\right] 
	& \temp &
	-\left[
	\begin{array}{cc}
	s & t \\ 
	0 & 0 \\ 
	\end{array}
	\right] 
	\left[
	\begin{array}{cc}
	z & u \\ 
	v & w \\ 
	\end{array}
	\right] ^{-1}
	\\ 
	\hline
	-\left[
	\begin{array}{cc}
	z & u \\ 
	v & w \\ 
	\end{array}
	\right]^{-1}
	\left[
	\begin{array}{cc}
	0 & p \\ 
	0 & q \\ 
	\end{array}
	\right] 
	& \temp &
	\left[
	\begin{array}{cc}
	z & u \\ 
	v & w \\ 
	\end{array}
	\right] ^{-1}
	\\ 
\end{array}
\right] .
\]
And then we have 
\[
\newcommand*{\temp}{\multicolumn{1}{|}{}}
X^{-1}BX = 
\left[\def\arraystretch{1.1}
\begin{array}{ccc}
	\left[
	\begin{array}{cc}
	0 & \alpha' \\ 
	0 & 0 \\ 
	\end{array}
	\right] 
	& \temp &
	\left[
	\begin{array}{cc}
	\sigma & \tau \\ 
	0 & 0 \\ 
	\end{array}
	\right] 
	\left[
	\begin{array}{cc}
	z & u \\ 
	v & w \\ 
	\end{array}
	\right] 
	\\ 
	\hline
	\left[
	\begin{array}{cc}
	z & u \\ 
	v & w \\ 
	\end{array}
	\right]^{-1}
	\left[
	\begin{array}{cc}
	0 & \lambda \\ 
	0 & \mu \\ 
	\end{array}
	\right] 
	& \temp & 
	\left[
	\begin{array}{cc}
	0 & 0 \\ 
	0 & 0 \\ 
	\end{array}
	\right] 
	\\ 
\end{array}
\right] ,
\]
where
\[
\left[
	\begin{array}{cc}
	0 & \alpha' \\ 
	0 & 0 \\ 
	\end{array}
\right]  =
\left[
	\begin{array}{cc}
	0 & \alpha \\ 
	0 & 0 \\ 
	\end{array}
\right]
 -\left[
	\begin{array}{cc}
	s & t \\ 
	0 & 0 \\ 
	\end{array}
\right]
\left[
	\begin{array}{cc}
	z & u \\ 
	v & w \\ 
	\end{array}
\right]^{-1}
\left[
	\begin{array}{cc}
	0 & \lambda \\ 
	0 & \mu \\ 
	\end{array}
\right] 
+ \left[
	\begin{array}{cc}
	\sigma & \tau \\ 
	0 & 0 \\ 
	\end{array}
\right]
\left[
	\begin{array}{cc}
	0 & p \\ 
	0 & q \\ 
	\end{array}
\right] 
\]

$\lambda, \mu, \sigma$ and $\tau$ cannot be $0$ at the same time, otherwise $(A, B)$ is decompsable. As such, $\alpha'$ can be reduced to $0$ because the free variables
$p$, $q$, $s$ and $t$ appear as linear terms in $\alpha'$.

If $\sigma\ne 0$ or $\tau\ne 0$, then 
$\left[
	\begin{array}{cc}
	\sigma & \tau \\ 
	0 & 0 \\ 
	\end{array}
\right]	
\left[
	\begin{array}{cc}
	z & u \\ 
	v & w \\ 
	\end{array}
	\right] $
can be reduced to 
$\left[
	\begin{array}{cc}
	1 & 0 \\ 
	0 & 0 \\ 
	\end{array}
\right]$. To further reduce 
$
\left[
	\begin{array}{cc}
	z & u \\ 
	v & w \\ 
	\end{array}
\right] ^{-1}
\left[
	\begin{array}{cc}
	0 & \lambda \\ 
	0 & \mu \\ 
	\end{array}
\right]$,
we are limited to those matrices
$\left[
	\begin{array}{cc}
	z & u \\ 
	v & w \\ 
	\end{array}
\right]$
which satisfy that
$\left[
	\begin{array}{cc}
	1 & 0 \\ 
	0 & 0 \\ 
	\end{array}
\right] 
\left[
	\begin{array}{cc}
	z & u \\ 
	v & w \\ 
	\end{array}
\right] = 
\left[
	\begin{array}{cc}
	1 & 0 \\ 
	0 & 0 \\ 
	\end{array}
\right]
$. It follows that $z=1$ and $u=0$. Thus in this case,
\[
\left[
	\begin{array}{cc}
	z & u \\ 
	v & w \\ 
	\end{array}
\right] ^{-1}
 =
\left[
	\begin{array}{cc}
	1 & 0 \\ 
	v & w \\ 
	\end{array}
\right] ^{-1} 
=
\left[
	\begin{array}{cc}
	1 & 0 \\ 
	-w^{-1}v & w^{-1} \\ 
	\end{array}
\right].
\]
And hence we have
\[
\left[
	\begin{array}{cc}
	z & u \\ 
	v & w \\ 
	\end{array}
\right] ^{-1}
\left[
	\begin{array}{cc}
	0 & \lambda \\ 
	0 & \mu \\ 
	\end{array}
\right] =
\left[
	\begin{array}{cc}
	0 & \lambda \\ 
	0 & w^{-1}(\mu - v\lambda) \\ 
	\end{array}
\right]. 
\]
If $\lambda\ne 0$ then 
$\left[
	\begin{array}{cc}
	0 & \lambda \\ 
	0 & w^{-1}(\mu - v\lambda) \\ 
	\end{array}
\right]$ 
can be reduced to 
$\left[
	\begin{array}{cc}
	0 & \lambda \\ 
	0 & 0 \\ 
	\end{array}
\right]$. Thus $(A, B)$ can be reduced to 
\[
\newcommand*{\temp}{\multicolumn{1}{|}{}}
\left(
\left[
\begin{array}{cccccc}
0 & 1 & \temp & 0 & 0 \\ 
0 & 0 & \temp & 0 & 0 \\ 
\cline{1-5}
0 & 0 & \temp & 0 & 0 \\
0 & 0 & \temp & 0 & 0 \\
\end{array}
\right],
\left[
\begin{array}{cccccc}
0 & 0 & \temp & 1 & 0 \\ 
0 & 0 & \temp & 0 & 0 \\ 
\cline{1-5}
0 & \lambda & \temp & 0 & 0 \\
0 & 0 & \temp & 0 & 0 \\
\end{array}
\right] 
\right),
\]
which is decomposable. Thus $\lambda$ must be 0. In this case, $\mu$ cannot be 0, otherwise $(A,B)$ is decomposable. 
So $\left[
	\begin{array}{cc}
	0 & \lambda \\ 
	0 & w^{-1}(\mu - v\lambda) \\ 
	\end{array}
\right]=
\left[
	\begin{array}{cc}
	0 & 0 \\ 
	0 & w^{-1}\mu \\ 
	\end{array}
\right],$
which can be reduced to 
$\left[
	\begin{array}{cc}
	0 & 0 \\ 
	0 & 1 \\ 
	\end{array}
\right]$.
And so $(A,B)$ can be reduced to 
\[
\newcommand*{\temp}{\multicolumn{1}{|}{}}
\left(
\left[
\begin{array}{cccccc}
0 & 1 & \temp & 0 & 0 \\ 
0 & 0 & \temp & 0 & 0 \\ 
\cline{1-5}
0 & 0 & \temp & 0 & 0 \\
0 & 0 & \temp & 0 & 0 \\
\end{array}
\right],
\left[
\begin{array}{cccccc}
0 & 0 & \temp & 1 & 0 \\ 
0 & 0 & \temp & 0 & 0 \\ 
\cline{1-5}
0 & 0 & \temp & 0 & 0 \\
0 & 1 & \temp & 0 & 0 \\
\end{array}
\right] 
\right),
\]
The fact that the above pair is indecomposable is left as an exercise to the interested readers. This verifies (4.3).

If $\sigma = 0$ and $\tau = 0$, then $\lambda$ and $\mu$ cannot be 0 at the same tine, otherwise $(A,B)$ is decomposable. And so 
$\left[
	\begin{array}{cc}
	z & u \\ 
	v & w \\ 
	\end{array}
	\right]^{-1}
	\left[
	\begin{array}{cc}
	0 & \lambda \\ 
	0 & \mu \\ 
	\end{array}
\right]$
can be reduced to 
$\left[
	\begin{array}{cc}
	0 & 1 \\ 
	0 & 0 \\ 
	\end{array}
\right]$. And hence in this case, $(A, B)$ can be reduced to:
\[
\newcommand*{\temp}{\multicolumn{1}{|}{}}
\left(
\left[
\begin{array}{cccccc}
0 & 1 & \temp & 0 & 0 \\ 
0 & 0 & \temp & 0 & 0 \\ 
\cline{1-5}
0 & 0 & \temp & 0 & 0 \\
0 & 0 & \temp & 0 & 0 \\
\end{array}
\right],
\left[
\begin{array}{cccccc}
0 & 0 & \temp & 0 & 0 \\ 
0 & 0 & \temp & 0 & 0 \\ 
\cline{1-5}
0 & 1 & \temp & 0 & 0 \\
0 & 0 & \temp & 0 & 0 \\
\end{array}
\right]
\right) ,
\]
which is decomposable. So this case cannot happen.

\textbf{Case 5.} Let 
\[
A = 
\left[
\begin{array}{ccccc}
0 & 0 & 0 & 0 \\ 
0 & 0 & 0 & 0 \\ 
0 & 0 & 0 & 0 \\
0 & 0 & 0 & 0 \\
\end{array}
\right].
\]
Then we have $\mathrm{Stab}(A) = \mathrm{GL}(4,\mathbb{F})$ and $\mathrm{Nil}C(A) = \mathrm{Mat}(4, \mathbb{F})$. Since the pair $(A,B)$ is indecomposable, the Jordan normal form theorem implies that 
$B$ is smiliar to the following matrix:
\[
\left[
\begin{array}{ccccc}
0 & 1 & 0 & 0 \\ 
0 & 0 & 1 & 0 \\ 
0 & 0 & 0 & 1 \\
0 & 0 & 0 & 0 \\
\end{array}
\right].
\]
This verifies (5.1). 
\end{proof}

The reduction technique used in this paper may be applied to pairs of higher degrees. It is clear from the theorem above that the set of orbits corresponding to indecomposable pairs
admits a cell decomposition, i.e., it is a disjoint union of affine spaces $\mathbb{F}^i$ where $i\ge 0$. We anticipate that this phenomenon still holds for pairs of higher degrees.

\vspace{0.5cm}
\textit{Email address}: \texttt{jiuzhao.hua@gmail.com}

\end{document}